\documentclass{article}
\usepackage{amsthm,amsfonts, amsbsy, amssymb,amsmath,graphicx}
\usepackage{graphics}
\usepackage[english]{babel}
\author{Vassily Olegovich Manturov\footnote{Bauman Moscow State Technical Unversity 2nd Baumanskaya St.5/1, Moscow, 105005, Russia, and Laboratory of Quantum Topology, Chelyabinsk State University, Brat'ev Kashirinykh street 129, Chelyabinsk 454001, Russia.} \footnote{\tt {vomanturov@yandex.ru}}
\footnote{Partially supported by Laboratory of Quantum Topology of
Chelyabinsk State University (Russian Federation government grant
14.Z50.31.0020) and by grants of the Russian foundation for Basic
Resarch, 13-01-00830,14-01-91161, 14-01-31288}}
\date{}

\theoremstyle{plain}
\newtheorem{thm}{Theorem}

\theoremstyle{dfn}
\newtheorem{dfn}{Definition}
\newtheorem{rk}{Remark}

\def\R{{\mathbb R}}

 \def\Z{{\mathbb Z}}
 \def\0{{\mathbbf 0}}
 \def\1{{\mathbbf 1}}

 \newcommand{\skcr}{\raisebox{-0.25\height}{\includegraphics[width=0.5cm]{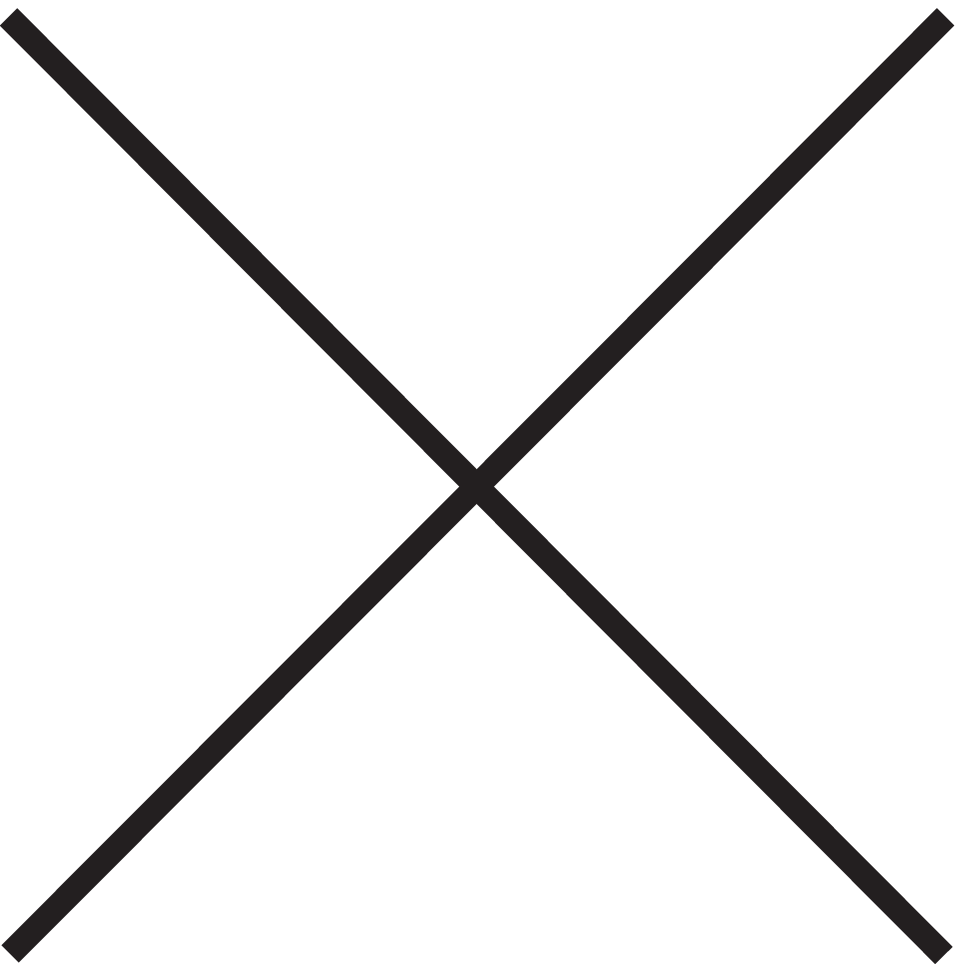}}}
 \newcommand{\skcrh}{\raisebox{-0.25\height}{\includegraphics[width=0.5cm]{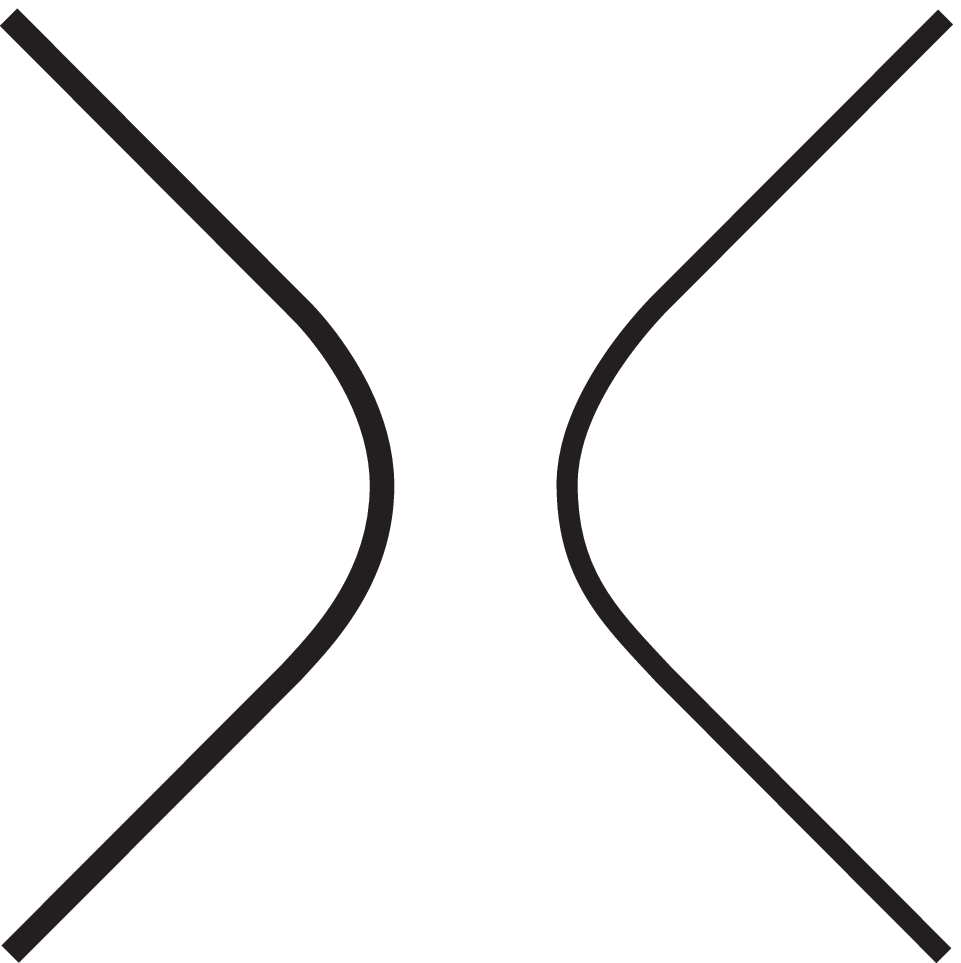}}}
\newcommand{\skcrv}{\raisebox{-0.25\height}{\includegraphics[width=0.5cm]{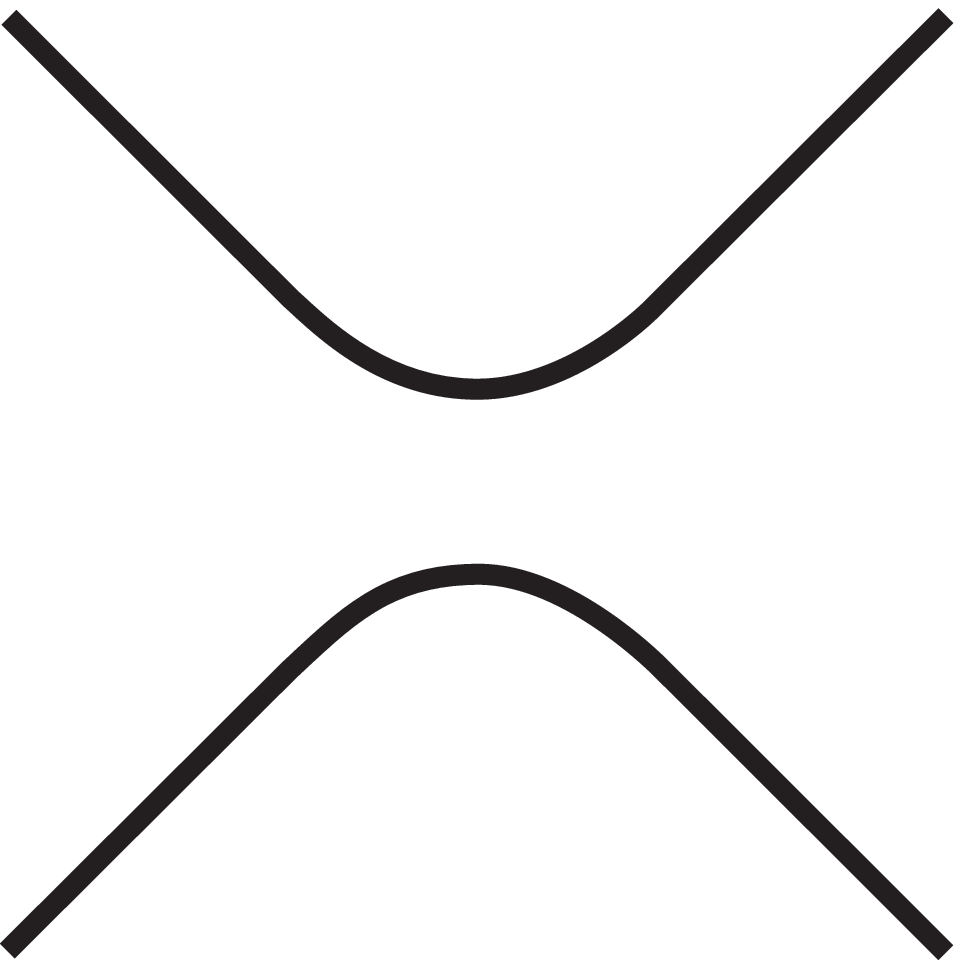}}}
\newcommand{\eps}{\varepsilon}

\title{Framed $4$-valent Graph Minor Theory II: Special Minors and New Examples}

\begin{document}

\maketitle

\begin{abstract}
In the present paper, we proceed the study of framed $4$-graph minor
theory initiated in \cite{ManStoneFlower} and justify the planarity
theorem for arbitrary framed $4$-graphs; besides, we prove analogous
results for embeddability in $\R{}P^{2}$.

{\bf Keywords:} graph, $4$-valent, minor, planarity, projective
plane, embedding, immersion, Wagner conjecture.
\end{abstract}

{\bf AMS MSC} 05C83,57M25, 57M27

\section{Introduction. Basic Notions}

Some years ago, a milestone in graph theory was established: as a
result of series of papers by Robertson, Seymour (and later joined
by Thomas) \cite{RS20} proved the celebrated Wagner conjecture
\cite{Wagner} which stated that if a class of graphs (considered up
to homeomorphism) is minor-closed (i.e., it is closed under edge
deletion and edge contraction), then it can be characterized by a
finite number of excluded minors. For a beautiful review of the
subject we refer the reader to L.Lov\'asz \cite{Lovasz}.

This conjecture was motivated by various evidences for concrete
natural minor-closed properties of graphs, such as knotless or
linkless embeddability in $\R^{3}$, planarity or embeddability in a
standardly embedded $S_{g}\subset \R^{3}$.

Framed $4$-valent graphs (see definition below) are a very important
class of graphs which arise as medial graphs of arbitrary graphs
drawn on $2$-surfaces. In some sense, they approximate arbitrary
graphs; in particular, a new proof of the Pontrjagin-Kuratowski
planarity criterion was recently found by Nikonov \cite{Nikonov}.
The two {\em smoothing} operations for framed $4$-valent medial
graph $M$ of a graph $\Gamma$ in a $2$-surface $\Sigma$ naturally
correspond to edge contraction and edge deletion of the graph
$G\subset \Sigma$ for which $M$ is the medial graph, see Fig.
\ref{smcontrdel}.

In the present paper, we proceed the study of minor closed
properties for framed $4$-valent graphs and go on proving theorems
that for some minor closed properties the number of minimal minor
obstructions is finite. Here a property is called {\em minor closed}
if whenever $P'$ is a minor of $P$ and $P$ possesses the property,
so does $P'$; a graph $Q$ is called a {\em minimal minor
obstruction} for some property if $Q$ does not possess the desired
property and all minors of $Q$ do; in \cite{ManStoneFlower} and in
the present paper we deal with various definitions of minors; they
lead to definitions of minor closed properties.

In \cite{ManStoneFlower}, we introduced the class of framed
$4$-graphs, $4$-valent graph minor theory and proved a planarity
criterion for framed $4$-graphs admitting a source-sink structure
(see below). Whenever drawing a framed four-valent graph on the
plane, we shall indicate its vertices by solid dots,
(self)intersection points of edges will be encircled, and the
framing is assumed to be induced from the plane: those half-edges
which are drawn opposite in $\R^{2}$ are thought to be opposite.
Half-edges  of a framed four-valent graph incident to the same
vertex are which are not {\em opposite}, are called {\em adjacent}.
According to the main theorem of \cite{ManStoneFlower}, the only
graph which was an obstruction to planarity for framed $4$-valent
graphs, is the $\Delta$-graph, see Fig. \ref{delta}.

\begin{figure}
\centering\includegraphics[width=200pt]{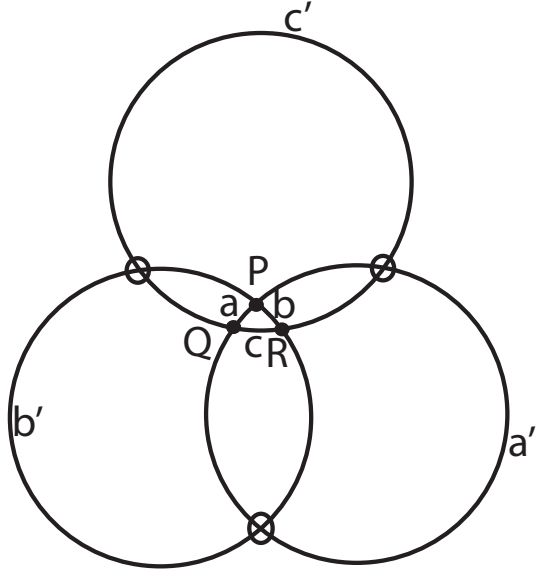} \caption{The
Graph $\Delta$} \label{Delta}
\end{figure}

This graph $\Delta$, in turn, successfully turns out to be the
unique obstruction (unique forbidden minor) for immersibility in
$\R^{2}$ with no more than $2$ crossings, for linkless embeddability
in $\R^{3}$ (in a proper sense). However, all these properties hold
upon an important obstruction imposed on framed $4$-graphs: {\em the
existence of a source-sink structure}.

We recall that a regular $4$-graph is called  {\em framed} if each
vertex of it is endowed with a {\em framing}: for the four emanating
edges, we indicate two pairs of {\em opposite} edges. Non-opposite
edegs are called {\em adjacent}. Besides graphs in the proper sense
we also allow $4$-graphs to have circular components.

At every vertex $V$ of a framed $4$-graph $\Gamma$, there are two
ways of pairing the four edges into two pairs of adjacent ones. For
each of these two pairings, we define the {\em smoothing} of
 $\Gamma$ at  $V$ as the graph obtained
by breaking $\Gamma$ at $X$, and pasting together one pair of the
four edges edges into one edge and the other pair of edges into the
other edge; $\skcr\to \skcrv$ and $\skcr\to \skcrh$.

\begin{figure}
\centering\includegraphics[width=220pt]{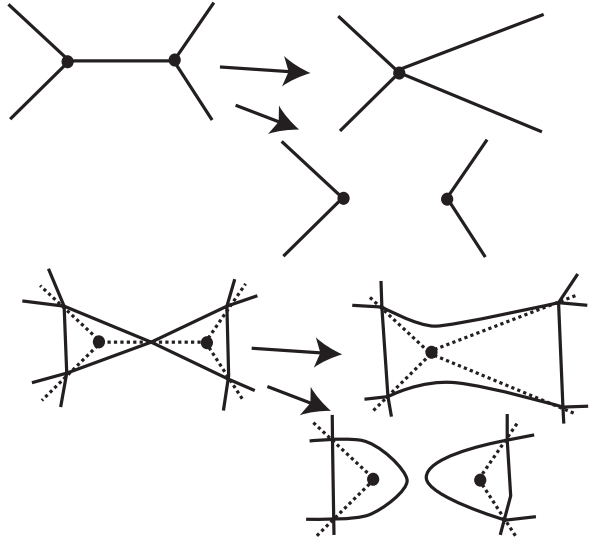} \caption{Edge
deletion and edge contraction yield smoothing} \label{smcontrdel}
\end{figure}

\begin{dfn}
A {\em source-sink structure} of a framed $4$-graph is an
orientation of all its circular edges together with an orientation
of all its non-circular edges such that at every crossing two
opposite half-edges are incoming, and the other two are emanating.
\end{dfn}

\begin{rk}
Whenever talking about an embedding or an immersion of a framed
$4$-graph into any $2$-surface we always assume its framing to be
preserved: opposite edges at every crossing should be locally
opposite on the surface.
\end{rk}

It follows obviously from the definition that for a connected framed
$4$-graph there exist no more than two source-sink structures:
starting with an orientation of an edge, we can orient all edges of
the corresponding  connected component.

Moreover, the {\em smoothing} operation at crossings, agrees with
the source-sink orientation: if the initial graph admits a
source-sink structure, then every smoothing of it inherits this
source-sink structure, see Fig.\ref{ssrespectsori}.

\begin{figure}
\centering\includegraphics[width=200pt]{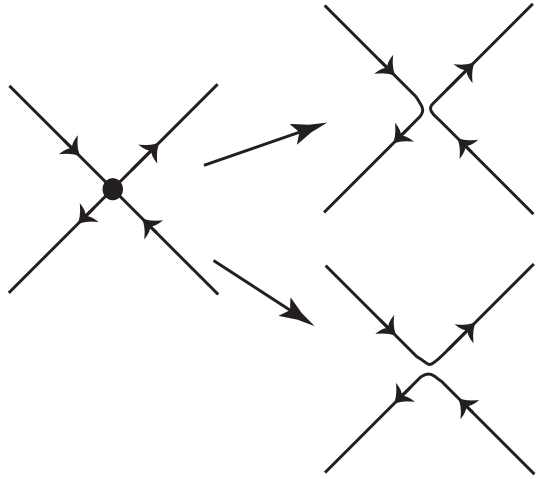} \caption{The
smoothing operation respects source-sink structure}
\label{ssrespectsori}
\end{figure}

\begin{dfn}
A framed $4$-valent graph $G'$ is a {\em minor} of a framed
$4$-valent graph $G$ if $G'$ can be obtained from $G$ by a sequence
of {\em smoothing} operations {\em (}$\skcr\to \skcrv$ and $\skcr\to
\skcrh${\em )} and deletions of connected components.
\end{dfn}

Now, let us look at the planarity problem. There exists a very
simple one-vertex graph $\Gamma$ having one vertex $X$ and two
half-edges $a$ and $b$ both adjacent to $X$ and each being opposite
to itself at $X$, see Fig. \ref{grX}.

\begin{figure}
\centering\includegraphics[width=200pt]{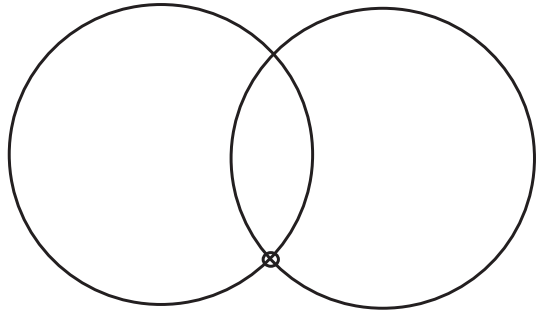} \caption{The graph
$\Gamma$} \label{grX}
\end{figure}

Note that $X$ admits no source-sink structures, thus, $X$ is not a
minor of any framed $4$-graph admitting a source-sink structure.

This framed $4$-graph is obviously non-planar: we have two cycles
with exactly one transverse intersection point \cite{ManVasConj}.

Besides, we can see in Fig. \ref{oneinone} that $\Gamma$ sits inside
$\Delta$: in Fig. \ref{oneinone} $\Gamma$ is drawn as a subgraph of
$\Delta$ in {\em red}.

\begin{figure}
\centering\includegraphics[width=200pt]{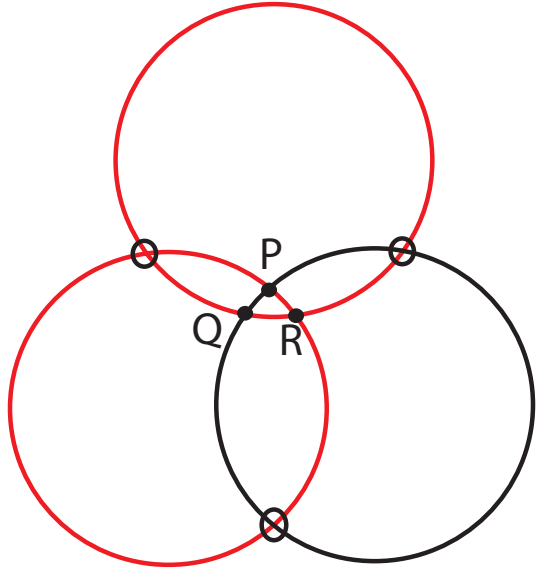} \caption{The
Graph $\Gamma$ inside the graph $\Delta$} \label{oneinone}
\end{figure}

So, $\Delta$ is a non-planar graph because of the graph $\Gamma$
``sitting inside'' it.

Thus, besides minors there should be another (more general) notion
expressing that $\Gamma$ ``sits inside $\Delta$''. The notion
described below will not be {\em local}. Let us now introduce {\em
$s$-minors} of framed $4$-graphs.

\begin{dfn} A framed $4$-graph $\Gamma'$ is an {\em $s$-minor} of
a framed $4$-graph $\Gamma$ if it can be obtained from $\Gamma$ by a
sequence of two operations:

\begin{enumerate}

\item Passing to a subgraph with all vertices of even valency and
deleting all vertices of valency $2$;

\item Removing connected components.

\end{enumerate}
\end{dfn}

\begin{rk}
By definition, if $P$ is a minor of $Q$, then $P$ is an $s$-minor of
$Q$; the inverse statement is wrong: $\Gamma$ is an $s$-minor of
$\Delta$, but is not a minor of $\Delta$.
\end{rk}

As in \cite{ManStoneFlower}, we shall use a way of coding framed
$4$-graphs by chord diagrams. Unlike \cite{ManStoneFlower}, we have
to encode all framed $4$-graphs, which will require a more general
notion of a {\em framed chord diagram}.

\begin{dfn}
By a {\em rotating circuit} of a connected framed $4$-graph not
homeomorphic to a circle we mean a surjective map $f:S^{1}\to
\Gamma$ which is a bijection everywhere except preimages of
crossings of $\Gamma$ such that at every crossing $X$ the
neighbourhoods $V(Y_{1})$
and $V(Y_{2})$ of the two preimages
$Y_{1},Y_{2}$ of $X$ belong to unions of {\em adjacent} half-edges
each. In other words, the circuit ``passes'' from a half-edge to a
non-opposite half-edge.

For a framed $4$-graph homeomorphic to the circle, the {\em circuit}
is a homeomorphism of the circle and the graph.

A circuit $f$ is called {\em good} at a vertex $X$ of $P$ if for the
two inverse images $Y_{1},Y_{2}\in S^{1}$ of $X$, the neighbourhoods
of the small segments $(Y_{1}-\eps,Y_{1}]$ and $(Y_{2}-\eps,Y_{2}]$
of the circle are taken by $f$ to a pair of {\em opposite}
half-edges at $X$.

Otherwise the rotating circuit is called {\em bad} at $X$.

The rotating circuit is {\em good} if it is good at every vertex.
\end{dfn}

Rotating circuits play a crucial role in the study of embeddings of
framed $4$-valent graphs, see
\cite{ManVasConj,ManDkld,ManHdlbg,Ilyu}.

An easy exercise (see, e.g. \cite{ManVasConj}) shows that {\em every
connected framed $4$-graph admits a rotating circuit}.

Usually, we shall denote a circuit by a small letter (say, $f$) when
we want to consider it as a map, and by a capital letter (say, $C$)
when we want to deal with its image as a subgraph.

\begin{rk}
It follows from the definition that if a connected non-circular
framed $4$-graph $P$ admits a source-sink structure then every
rotating circuit of $P$ is good at every vertex.

The opposite statement is true as well; moreover, if there exists
one rotating circuit which is good at every vertex, then the framed
$4$-graph $P$ admits a source-sink structure.
\end{rk}

\begin{dfn}
By a {\em chord diagram} we mean either an oriented circle {\em
(}{\em empty chord diagram}{\em )} or a cubic graph $D$ consisting
of an oriented cycle {\em (}the {\em core}{\em )} passing through
all vertices of $D$ such that the complement to it is a disjoint
union of edges {\em (}{\em chords}{\em )} of the diagram.

A chord diagram is {\em framed} if every chord of it is endowed with
a {\em framing} $0$ or $1$.
\end{dfn}

\begin{dfn}
We say that two chords $a,b$ of a chord diagram $D$ are {\em linked}
if the ends of the chord $b$ belong to two different components of
the complement $Co\backslash \{ a_{1},a_{2}\}$ to the endpoints of
$a$ in the core circle $Co$ of $D$.
\end{dfn}

Having a circuit $C$ of a framed connected $4$-graph $G$, we define
the framed chord diagram $D_{C}(G)$, as follows. If $G$ is a circle,
then $D_{C}(G)$ is empty. Otherwise, think of $C$ as a map
$f:S^{1}\to D$; then we mark by points on $S^{1}$ preimages of
vertices of $G$. Thinking of $S^{1}$ as a core circle and connecting
the preimages by chords, we get the desired cubic graph.

The framing of good vertices is set to be equal to $0$, the framing
of bad vertices is set to be equal to $1$.

\begin{rk}
{\em (}Framed{\em )} chord diagrams are considered up to
combinatorial equivalence.
\end{rk}

The opposite operation (of restoring a framed $4$-graph  from a
chord diagram) is obtained by removing chords from the chord diagram
and approaching two endpoints of each chord towards each other. For
every chord, we create a crossing, and for every chord with framing
zero, we create a small twist,
 as shown in Fig. \ref{CDgrF}.

\begin{figure}
\centering\includegraphics[width=200pt]{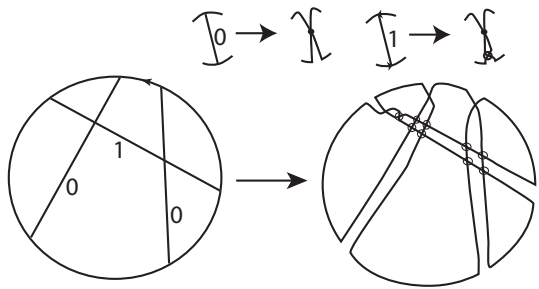}
\caption{Restoring a framed $4$-graph from a chord diagram}
\label{CDgrF}
\end{figure}

\begin{dfn}
A (framed) chord diagram $D'$ is called a {\em subdiagram} of a
chord diagram $D$ if $D$ can be obtained from $D$ by deleting some
chords and their endpoints (with framing respected).
\end{dfn}

It follows from the definition that the removal of a chord from a
framed chord diagram results in a smoothing of a framed $4$-graph.
Consequently, if $D'$ is a subdiagram of $D$, then the resulting
framed $4$-graph $G(D')$ is a {\em minor} of $G(D)$.

\section{The Planarity Criterion}

Note that planarity is a minor closed and $s$-minor closed property;
thus, it makes sense to look for minimal planarity obstructions.

When dealing with all framed $4$-graphs (not necessarily admitting a
source-sink structure), we obtain the following
\begin{thm}
A graph $P$ is non-planar if and only if it admits either $\Gamma$
or $\Delta$ as a minor.

Alternatively, $P$ is non-planar if and only if it admits $\Gamma$
as an $s$-minor.\label{theo}
\end{thm}

\begin{proof}
This proof goes along the lines of \cite{ManVasConj}: the second
statement of the theorem actually repeats the main statement of
\cite{ManVasConj}: a framed $4$-graph is non-planar iff it contains
two cycles sharing no edges and having exactly one transverse
intersection point. These two cycles form exactly an $s$-minor
$\Gamma$ inside $P$.

Let us now prove the first statement of the theorem. If a graph $P$
admits a source-sink structure then it admits $\Delta$ as a minor,
as proved in \cite{ManStoneFlower}. Otherwise, let us construct a
rotating circuit and a chord diagram of $P$. by definition, this
chord diagram will have at least one chord of framing one. This
means that $G(D_{1})$ is a minor of $P$, where $D_{1}$ is the chord
diagram with the unique chord of framing $1$. But one can easily see
that $G(D_{1})$ is isomorphic to $\Gamma$.

 This exactly means that $P$ contains $\Gamma$ as a
minor.
\end{proof}

\section{Checkerboard Embeddings and $\R{}P^{2}$}

\begin{dfn}
An embedding of a graph $P$ in a $2$-surface $\Sigma$ is {\em
cellular} if the complement $\Sigma\backslash P$ is a union of
$2$-cells.
\end{dfn}

When talking of embeddings, one usually deals with cellular ones.
For example, when defining the minimal embedding genus for a given
graph $P$, one certainly means the genus of a cellular embedding.
Nevertheless, the {\em cellular embeddability} into a surface of a
given genus $g$ is not a minor closed property (it is not so for
arbitrary graphs and minors in the usual sense; neither it is so for
framed $4$-graphs); the reason is that if $P$ is embeddable into
$\Sigma$, then it yields an embedding of any minor $P'$ of $P$ into
$\Sigma$; however, it may well happen that the complement to the
image $P$ is a union of $2$-cells, whence the complement to the
image of $P'$ is not.

Thus, we shall not restrict ourselves to just cellular embdeddings;
thus, for instance, every {\em planar} framed $4$-graph is
embeddable into any $2$-surface.

Having a framed $4$-graph $P$ and a $2$-surface $\Sigma$, we may
consider embeddings of $P$ into $\Sigma$. Among all embeddings, we
draw special attention to {\em checkerboard embedding}.

\begin{dfn}
A {\em checkerboard embedding} $f:P\to \Sigma$ is an embedding such
that the connected components of the complement $\Sigma\backslash P$
can be colored in black and white in a way such that every two
components sharing an edge have different colours.
\end{dfn}

One can easily see that checkerboard embeddability into any fixed
$2$-surface $\Sigma$ is a minor closed property: if the complement
to the image of a framed $4$-graph $P$ is checkerboard colourable,
then so is the complement to the image of $P'$, where $P'$ is
obtained from $P'$ by a smoothing at a vertex. Certainly, the
connected components to the image of $P'$ might not be homeomorphic
to $2$-cells.

For more about checkerboard embeddings of graphs, see \cite{FM1}.

Note that in the case when $\Sigma$ is $\R^{2}$ (or $S^{2})$), all
embeddings are checkerboard. It can be shown (\cite{ManDkld}) that
an embedding is checkerboard if and only if the image of the graph
viewed as an element of $H_{1}(\Sigma, \Z_{2})$ is zero.

\begin{thm}[\cite{ManDkld}]
If a framed $4$-graph $P$ admits a source-sink structure and a
 cellular checkerboard embedding
into a closed $2$-surface $\Sigma$, then $\Sigma$ is orientable. If
$P$ admits no source-sink structure but admits a cellular
checkerboard embedding into $\Sigma$ then $\Sigma$ is not
orientable.\label{ori}
\end{thm}

It turns out (see \cite{ManDkld}) that checkerboard embeddings are
very convenient to deal with: they lead to a splitting of the
surface into the {\em black part} and the {\em white part}, as
follows.

Given a framed $4$-graph $P$; let us consider a rotating circuit $C$
of it. A checkerboard embedding $g:P\to \Sigma$ leads to the
composite map $f\circ g:S^{1}\to \Sigma$; this map is bijective
everywhere except those points mapped to images of crossings of $P$.
At every crossing, we can slightly deform the map $f\circ g$ to get
an embedding.

\begin{figure}
\centering\includegraphics[width=200pt]{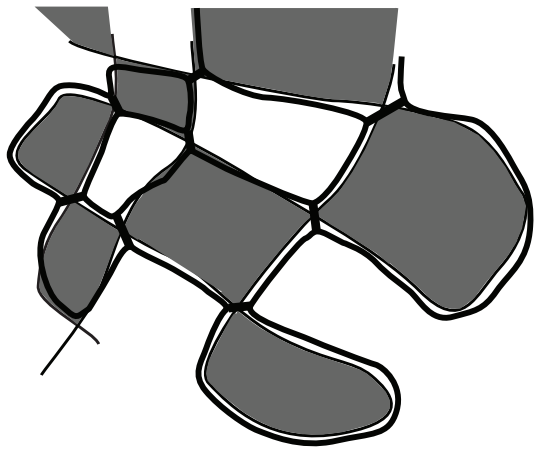} \caption{A
smoothed checkerboard embedding} \label{smoothinga}
\end{figure}

Denote the resulting embedding by $f'$. Note that $f(S^{1})$ splits
the surface $\Sigma$ into the {\em black part} $\Sigma_{B}$ and the
white part $\Sigma_{W}$. Moreover, chords of the chord diagram
$D_{C}(P)$ can be naturally thought of as ``black'' ones and
``white'' ones: small segments in neighbourhoods of vertices shown
in Fig. \ref{smoothinga} can be thought of as images of chords of
the chord diagram.

Now, if the surface $\Sigma$ is homeomorphic to $\R{}P^{2}$ then one
of the two parts $\Sigma_{B},\Sigma_{W}$ is homeomorphic to the
M\"obius band, and the other part is homeomorphic to the disc.

This means that the chord diagram $D_{C}(P)$ has a very specific
form. Namely, in \cite{ManDkld}, the following theorem is proved.

\begin{thm}
A framed $4$-graph $P$ is embeddable in $\R{}P^{2}$ if and only if
the for some  rotating circuit $C$ there is a way to split all
chords of $D_{C}(P)$ into two families in such a way that the
resulting subdiagrams $D_{1}$ and $D_{2}$ are as follows:
\begin{enumerate}

\item All chords of $D_{1}$ having framing $0$ are pairwise unlinked;

\item All chords of $D_{2}$ of framing $1$ are pairwise linked; all
chords of framing $0$ of $D_{2}$ are pairwise unliked with all other
chords of $D_{2}$.

\end{enumerate}
\label{thy}
\end{thm}

In \cite{ManDkld}, it is shown that if the condition of Theorem
\ref{thy} holds for {\em some} rotating circuit of $P$ then it holds
for {\em any} rotating circuit of $P$. The point is that if the
condition of Theorem \ref{thy} holds for {\em some} rotating
circuit, then this means that one family of chords leads to the
subdiagram which corresponds to the disc, and the other family leads
to the subdiagram which corresponds to the M\"o{}bius band. For more
details, see \cite{ManDkld}.

Now, let us denote by ${\hat D}$ the framed chord diagram with two
unlinked chords of framing $1$; let $\Gamma_{1}$ be the
corresponding framed $4$-graph.

Now, we are ready to formulate the main theorem of the present
section.
\begin{thm}
A framed $4$-graph $P$ is not checkerboard-embeddable in $\R{}P^{2}$
if and only if it contains one of the subgraphs $\Delta,\Gamma_{1}$
as a minor.

More precisely, if $P$ admits a source-sink structure then
checkerboard-embeddability in $\R{}P^{2}$ is equivalent to
checkerboard embeddability into $\R{}^{2}$. If $P$ does not admit a
source-sink structure then the only obstruction to checkerboard
embeddability into $\R{}P^{2}$ is the existence of $\Gamma_{1}$ as a
minor.
\end{thm}

\begin{proof}
Note that if a checkerboard embedding of $\iota:P\to \R{}P^{2}$ is
not cellular, then one component of the complement
$\R{}P^{2}\backslash \iota(P)$ is homeomorphic to the M\"o{}bius
band without boundary; removing this M\"{o}bius band and pasting its
boundary component by a disc, we se that  $P$ is actually {\em
planar}.

Thus, according to Theorem \ref{ori}, if $P$ admits a checkerboard
embedding to $\R{}P^{2}$ then $P$ is either planar or it does not
admit a source-sink structure.

We know that if $P$ admits a source-sink structure then $\Delta$ is
the only planarity obstruction.

Now, assume $P$ does not admit a source-sink structure. Take a
rotating circuit $C$ and consider a framed chord diagram $D_{C}(P)$.

According to our assumption, this chord diagram has at least one
chord of framing $1$.

Let us look at the obstruction from Theorem \ref{thy}. We shall try
to split all chords into two families $D_{1}$ and $D_{2}$ and see
when it is impossible.

Let $H$ be the following chord diagram graph: vertices of $H$ are in
one-to-one correspondence with chords of $D_{C}(P)$, and two
vertices are connected by an edge if either one of the corresponding
chords has framing zero and the chords are linked or both chords
have framing $1$ and they are unlinked.

It is easy to see that the existence of two families $D_{1}$ and
$D_{2}$ as in Theorem \ref{thy} means exactly that $H$ is bipartite.
The obstruction for $H$ to be bipartite is an existence of an odd
cycle. Consider such a cycle with the smallest possible number of
vertices; denote them subsequently by $v_{1},\dots, v_{2k+1}$ and
the corresponding chords $c_{1},\dots, c_{2k+1}$ of $D_{C}(P)$. If
all these chords have framing $0$, then the corresponding diagram
has a $(2k+1)$-gon as a subdiagram; hence, $P$ has $\Delta$ as a
minor.

Now, assume there is at least one chord of framing $1$ in the cycle
$v_{1}, \dots, v_{2k+1}$. If there are $3$ chords of framing one
among $c_{j}$, one can easily find a shorter cycle with the same
property. Thus, there are either exactly two chords of framing $1$
or exactly one chord of framing $1$. If we have two chords of
framing $1$ and they are linked, we may find a shorter cycle in $H$
with the required property.

If we have two unlinked chords of framing $1$, then they form a
subdiagram $D'$ such that the minor corresponding to $D'$ is
isomorphic to $\Gamma_{1}$.

Thus, it remains to consider the case when we have exactly one chord
of framing $1$ in our odd cycle. Without loss of generality, assume
the only chord of framing $1$ is $c_{1}$; consider the chords
$c_{2}$ and $c_{3}$ are linked chords of framing $0$.

Consider the framed $4$-graph $P_{2k+1}$ corresponding to this
cycle; it is a minor of $P$. Let us show now that $\Gamma_{1}$ is a
minor of $P_{2k+1}$. First, we shall show that $P_{2k-1}$ is a minor
of $P_{2k+1}$ for every $k\ge 1$ in a way similar to the proof of
Theorem \ref{theo}. Finally, we shall show that $\Gamma_{1}$ is a
minor of $P_{3}$. Denote the chords of $P_{2k+1}$ by the same
letters as those of $P$.

At each chord of framing $0$, there are two ways of smoothing of the
corresponding vertex: one way gives rise to the graph corresponding
to the subdiagram obtained from the initial diagram by deleting the
chord, and the other one.

 Let us change the rotating circuit of
$P_{2k+1}$ at vertices corresponding to $v_{2},v_{3}$.

By abuse of notation, denote by $c_{j}$ the chord of the new circuit
 corresponding to the vertex which corresponds to $c_{j}$ in the
 initial circuit. Then we see that in the chain
formed by all chords except $c_{2},c_{3}$ the incidences changes
only for the pair $(c_{1},c_{4})$; thus, wet a $(2k-1)$-gon.

 Here we see the $(2k-1)$-gon which shows that $P_{2k-1}$ is a
minor of $P_{2k+1}$.

Finally, if we look at $P_{3}$ and change the circuit at the vertex
$v_{1}$ as shown in Fig. \ref{tongbei2}, we see that $P_{3}$
contains $\Gamma_{1}$ as a minor.

\begin{figure}
\centering\includegraphics[width=200pt]{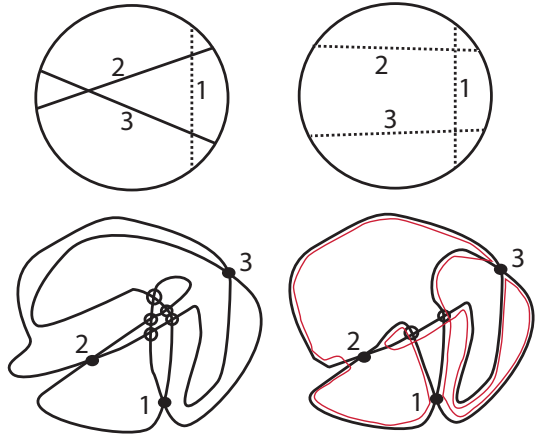}
\caption{$P_{2k-1}$ is a minor of $P_{2k+1}$} \label{tongbei2}
\end{figure}

This completes the proof of the Theorem.

\end{proof}

I am grateful to D.P.Ilyutko for various fruitful consultations and
useful remarks.

\end{document}